\documentclass[12pt]{amsart}
\usepackage{amsmath,amssymb,mathtools,mathabx,rotating}
\usepackage{graphicx}
\usepackage{color}
\usepackage{epstopdf}

\newtheorem{thm}{Theorem}
\newtheorem{exam}{Example}

\newtheorem{cor}[thm]{Corollary}

\newtheorem{lem}[thm]{Lemma}
\newtheorem{prop}[thm]{Proposition}
\newtheorem{defn}{Definition}
\newtheorem{rmk}[thm]{Remark}

\begin{document}
 
\title[$\epsilon$-noncrossing partitions]{$\epsilon$-noncrossing partitions and\\ cumulants in free probability}

\vspace{1cm}

\author{Kurusch Ebrahimi-Fard}
\address{Department of Mathematical Sciences, 
		Norwegian University of Science and Technology (NTNU), 
		7491 Trondheim, Norway.
		{\tiny{On leave from UHA, Mulhouse, France.}}}
         	\email{kurusch.ebrahimi-fard@math.ntnu.no, kurusch.ebrahimi-fard@uha.fr}         
        		\urladdr{https://folk.ntnu.no/kurusche/}

\author{Fr\'ed\'eric Patras}
\address{Univ.~de Nice,
		Labo.~J.-A.~Dieudonn\'e,
         	UMR 7351, CNRS,
         	Parc Valrose,
         	06108 Nice Cedex 02, France.}
		\email{patras@unice.fr}
		\urladdr{www-math.unice.fr/$\sim$patras}

\author{Roland Speicher}
\address{Universit\"at~des Saarlandes,
		Fachrichtung Mathematik,
		Postfach 151150,
		66041 Saarbr\"ucken, Germany}
		\email{speicher@math.uni-sb.de}
		\urladdr{https://www.math.uni-sb.de/ag/speicher/index.html}


\begin{abstract}
Motivated by recent work on mixtures of classical and free probabilities, we introduce and study the notion of $\epsilon$-noncrossing partitions. It is shown that the set of such partitions forms a lattice, which interpolates as a poset between the poset of partitions and the one of noncrossing partitions. Moreover, $\epsilon$-cumulants are introduced and shown to characterize the notion of $\epsilon$-independence.
\end{abstract}

\maketitle

\begin{quote}
\end{quote}


\section{Introduction}
\label{sec:intro}

{
The notion of $\epsilon$-independence was introduced by M\l{}otkowski in \cite{Mlotkowski_2004} and developed further in the recent work \cite{speichWyso_2016}. This concept describes a mixture of classical and free independence; thus on one hand, its investigation allows a very general frame which should include classical and free probability at the same time. On the other hand, this twilight zone between classical and free yields also a wealth of new examples for non-commutative forms of independences and stimulate new directions for investigations. For example, the determination of the quantum symmetries behind this $\epsilon$-independence in \cite{speicherweber_2016} gave not only new examples of quantum groups, but led also to the discovery of $\epsilon$-versions of spheres. A crucial ingredient in \cite{speicherweber_2016} was the reformulation of $\epsilon$-independence from \cite{speichWyso_2016} in terms of free cumulants. However, though this formulation gave a useful description for the calculation of mixed moments in $\epsilon$-independent variables it was lacking a more general conceptual approach; which should rest on a general notion of $\epsilon$-cumulants. Here we will provide this missing part of the theory. Namely, we introduce and explore the notion of $\epsilon$-noncrossing partitions. It will be shown that they form a lattice, which interpolates as a poset between the poset of partitions and the one of noncrossing partitions. Then we will use these $\epsilon$-noncrossing partitions to define, in the standard way, $\epsilon$-cumulants. Our main theorem will then be that $\epsilon$-independence is equivalent to the vanishing of mixed $\epsilon$-cumulants. This includes as special cases the corresponding characterizations of classical independence by classical cumulants and of free independence by free cumulants. The calculation rule for mixed moments for $\epsilon$-independent variables from \cite{speichWyso_2016} is then also an easy consequence of this.
}

\medskip

The article is organized as follows. In the next section we review the notion of $\epsilon$-independence and fix notation. The third section introduces $\epsilon$-noncrossing partitions. The fourth section shows that they form a lattice. The fifth section proves various identities satisfied by $\epsilon$-cumulants. The last section shows the equivalence between $\epsilon$-independence and the vanishing of $\epsilon$-cumulants.

\smallskip
{\bf{Acknowledgements}}: The authors acknowledge support from the grant CARMA
ANR-12-BS01-0017, from the CNRS-CSIC PICS Program ``Alg\`ebres
de Hopf combinatoires et probabilit\'es non commutatives'' and from the ERC-Advanced grant ``Noncommutative distributions in free probability'' (grant no.~339760).


\section{Preliminaries}
\label{sect:prelim}

In what follows $\mathbb K$ denotes a ground field of characteristic zero over which all algebraic structures are considered. The set of partitions of $[n]:=\{1,\ldots,n\}$ is denoted $\mathcal{P}_n$ and forms a lattice \cite{Stanley}. Recall that a partition $\Gamma=\Gamma_1 \coprod \cdots \coprod \Gamma_k$ of $[n]$ consists of a collection of (non-empty) subsets $\Gamma_i \subset [n]$, called blocks, which are mutually disjoint, $\Gamma_i \cap \Gamma_j=\emptyset$ for all $i \neq j$, and whose union, $\cup_{i=1}^k \Gamma_i$, is $[n]$. 
The number of blocks of $\Gamma$ and the number of elements in a specific block $\Gamma_i$ of $\Gamma$  are denoted by $|\Gamma|:=k$ and $|\Gamma_i|$, respectively. If $X$ is a subset of $[n]$, we write $\Gamma_{|X}$ for the restriction of $\Gamma$ to $X$: $\Gamma_{|X}=(\Gamma_1\cap X)\coprod \cdots \coprod (\Gamma_k\cap X).$

The lattice $\mathcal{P}_n$ has a partial order induced by refinement, i.e., $\Gamma \leq \Theta$ if $\Gamma$ is a finer partition than $\Theta$. The partition denoted by $\hat{1}_n$ consists of a single block, i.e., $|\hat{1}_n|=1$. It is the maximal element in $\mathcal P_n$. The partition $\hat{0}_n$, on the other hand, is the one with $n$ blocks and is the minimal element in $\mathcal P_n$. For example, for $n=3$ the maximal and minimal elements in $\mathcal P_3$ are $\{1,2,3\}$ and $\{1\}\coprod\{2\}\coprod\{3\}$, respectively. The other partitions of order three are: $\{1\}\coprod\{2,3\}$; $\{1,2\}\coprod\{3\}$; $\{1,3\}\coprod\{2\}$. A partition $\Gamma$ is called {\it{noncrossing}} if and only if there are no elements $i<j<k<l$ in $[n]$, such that $\{i,k\}$ and $\{j,l\}$ belong to two disjoint blocks of the partition $\Gamma$. Noncrossing partitions form a lattice denoted $\mathcal{NC}_n$. For $n=1,2,3$ all partitions are noncrossing. For $n=4$ the partition $\{1,3\}\coprod\{2,4\}$ is not in $\mathcal{NC}_4$. 

\smallskip 

Recall that a {\it{noncommutative probability space}}, $(A,\phi)$, consists of an associative $\mathbb K$-algebra $A$ with unit $1_A$ together with a unital linear functional $\phi : A \to  {\mathbb K}$ (so that $\phi(1_A)=1$). For background, definitions and terminology of free probability the reader is referred to \cite{SpeicherNica}. 

In the following $I$ denotes a finite or infinite index set, and for $i,j \in I$, $\epsilon_{ij}=\epsilon_{ji} \in \{0,1\}$. 
{For the definition of $\epsilon$-independence the values $\epsilon_{ii}$ on the diagonal do not play a role. They will only become important later, when we are going to define $\epsilon$-noncrossing partitions and $\epsilon$-cumulants.
Notice that for this we will not assume that $\epsilon_{ii} =0$, but instead will allow for all possible combinations of $0$ and $1$ for the diagonal values.}

The subset $I_n^\epsilon \subset I$ consists by definition of $n$-tuples $\mathbf{i}=\{i(1),\ldots,i(n)\}$ which satisfy that if $i(k) = i(l)$ for $1 \le k < l \le n$, then there exists a $k < p < l$ such that $ i(k) \neq i(p)$ and $\epsilon_{i(p) i(k)} =0.$

For any linear functional $\psi:\bigoplus_{l=0}^\infty A^{\otimes l}\to\mathbb K$, any sequence ${\mathbf a}=(a_1,\dots,a_n)$ of elements of $A$ and any partition $\Gamma =\Gamma_1\coprod \cdots \coprod \Gamma_k \in \mathcal{P}_n$, we write 
\begin{equation}
\label{def:tensormaps}
	\psi^\Gamma ({\mathbf a})=\psi^\Gamma (a_1,\dots,a_n):=\prod\limits_{i=1}^k\psi (a_{\Gamma_i}),
\end{equation}
where for $X=\{x_1,\dots,x_i\}\subset [n]$ with $x_i<x_{i+1}$ for $ i<n$:
 $$
 	\psi(a_X):=\psi(a_{x_1}\otimes\dots \otimes a_{x_i}).
$$


We now consider unital subalgebras $(A_i)_{i \in I}$ in $A$, and define the notion of $\epsilon$-{\it{independence}}. It is assumed that these algebras share the unit $1_A$ with $A$. However, sometimes we will have to distinguish notationally between the units according to the algebra they are considered to belong to, and will then write $1_i$ for the unit of the subalgebra $A_i$.

\begin{defn} \label{epsilon-independence}
Let $(A,\phi)$ be a noncommutative probability space. The subalgebras $A_i \subset A$, $i \in I$, are {\rm{$\epsilon$-independent}} if and only if 
\begin{enumerate}
	\item[i)] \label{item1} the algebras $A_i$ and $A_j$ commute  {for all $i\not=j$ for which $\epsilon_{ij}=1$.}
	
\item[ii)] \label{item2} for $({i(1)}, \ldots,{i(n)}) \in  I_n^\epsilon$ and $(a_1,\ldots,a_n) \in A_{i(1)} \times \cdots \times A_{i(n)}$, such that $\phi(a_k)=0$ for all $1 \le k \le n$, it follows that $\phi(a_1 \cdots a_n)=0$.
\end{enumerate}
\end{defn}

Let us consider the set of maps from $[n]$ to $I$, $I^{[n]}$, as a set of decorations of the integers. An element $d$ of $I^{[n]}$ is called simply a decoration and is represented by the sequence of values $(d(1),\dots,d(n))$. For a given $d \in I^{[n]}$, a {\it{$d$-decorated partition}} $\Gamma$ of $[n]$ is a partition $\Gamma =\Gamma_1\coprod \cdots \coprod \Gamma_k$ where the elements of block $\Gamma_i$ inherit the  $d$-decorations. The set of $d$-decorated partitions of $[n]$ is denoted by $\mathcal P^d_n$. For notational clarity, we will write $\bar\Gamma = \bar\Gamma_1 \coprod \cdots \coprod\bar\Gamma_k$ for the underlying undecorated partition of $[n]$. The $d$-decorated partitions are represented indexing the integer $i$ by $d(i)$. For example, in $\mathcal P^d_3$ the $d$-decorated partitions are denoted: $\{1_{d(1)},2_{d(2)},3_{d(3)}\}$; $\{1_{d(1)}\}\coprod  \{2_{d(2)}\}\coprod\{3_{d(3)}\}$; $\{1_{d(1)}\}\coprod\{2_{d(2)},3_{d(3)}\}$; $\{1_{d(1)},2_{d(2)}\}\coprod\{3_{d(3)}\}$; $\{1_{d(1)},3_{d(3)}\}\coprod\{2_{d(2)}\}$. 

The $n$-th symmetric group, $S_n$, is acting on decorated partitions as follows. For any $d$-decorated partition $\Gamma$ of $[n]$, the action of $\sigma\in S_n$ is the usual one on the underlying partition (i.e., the one induced by the action on elements):
$$
	\overline{\sigma(\Gamma)}:={\sigma(\bar\Gamma_1)}\coprod \cdots\coprod{\sigma(\bar\Gamma_k)}.
$$
However, in the $d$-decorated case the integer $\sigma(i)$ is now decorated by $d(i)$. That is, the decoration, written $\sigma(d)$, of $\sigma(\Gamma)$ is given by $\sigma(d)(i):=d(\sigma^{-1}(i)).$ For example, under the transposition $\tau_2=(2,3)$ the $d$-decorated partitions in $\mathcal P^d_3$ in the above example are mapped to: $\{1_{d(1)},3_{d(2)},2_{d(3)}\}$; $\{1_{d(1)}\}\coprod\{3_{d(2)}\}\coprod\{2_{d(3)}\}$; $\{1_{d(1)}\}\coprod\{3_{d(2)},2_{d(3)}\}$; $\{1_{d(1)},3_{d(2)}\}\coprod\{2_{d(3)}\}$; $\{1_{d(1)},2_{d(3)}\}\coprod\{3_{d(2)}\}$.

Let $X=\{x_1,\dots,x_n\}$ (where $x_i<x_{i+1}$) be an arbitrary finite subset of the integers decorated by $d\in I^X$. We call {\it{standardization of $d$}} and write $st(d)$ the decoration of $[n]$ defined by $st(d)(i):=d(x_i)$. Writing $\mu$ for the map $\mu:X\to [n]$, $x_i \mapsto \mu(x_i):=i$, a $d$-decorated partition $\Gamma$ of $X$ gives rise to a $st(d)$-decorated partition $\mu(\Gamma)$ of $[n]$ called the {\it{standardization of $\Gamma$}} and written $st(\Gamma)$.


\section{$\epsilon$-noncrossing partitions}
\label{sect:encp}

With the aim to simplify the presentation of the material, an assertion such as ``consider a decorated partition $\Gamma$'' shall mean from now on and if not stated otherwise, ``for $d \in I^{[n]}$ consider a $d$-decorated partition $\Gamma=\Gamma_1\coprod \cdots \coprod \Gamma_k \in \mathcal P^d_n$''.

{We will now define the notion of $\epsilon$-noncrossing partitions. For this the value of $\epsilon$ on the diagonal will become important; different choices for the values on the diagonal will lead to different kinds of  $\epsilon$-cumulants for the description of the same $\epsilon$-independence. The idea is that $\epsilon_{d(i)d(j)}$ tells us whether we can use $a_ia_j=a_ja_i$, for $a_i\in A_{d(i)}$ and $a_j\in A_{d(j)}$, in order to bring a decorated partition in a noncrossing form; however, if $d(i)=d(j)$ then $a_ia_j=a_ja_i$ would mean that elements in $A_{d(i)}$ have to commute, which we do not want to assume. By putting $\epsilon_{d(i)d(j)}=1$ we will thus only enforce such a commutation between $a_i$ and $a_j$ if they belong to two different blocks in our partition.
}

\begin{defn}
Let $\Gamma \in \mathcal P^d_n$ be a decorated partition such that $\epsilon_{d(i)d(i+1)}=1$ and that either

(a) $d(i)\not=d(i+1)$  

or 

(b) $d(i)=d(i+1)$ and $i$ and $i+1$ belong to different blocks of $\Gamma$. 

Then we say that the transposition $\tau_i:=(i,i+1)$ is an {\rm{allowed move}} between the $d$-decorated partition $\Gamma$ and the $\sigma(d)$-decorated partition $\Gamma' := \tau_i(\Gamma)$. An element $\sigma\in S_n$ which can be written as a sequence $\sigma=\tau_{i_k}\cdots \tau_{i_1}$ of allowed moves (i.e., $\tau_{i_j}$ is an allowed move from $\tau_{i_{j-1}}\cdots\tau_{i_1}(\Gamma)$ to $\tau_{i_{j}}\cdots\tau_{i_1}(\Gamma)$) is called an {\rm{allowed permutation}} of the partition~$\Gamma$.
\end{defn}

{Note that excluding $\tau_i$ when $\epsilon_{d(i)d(i+1)}=1$, $d(i)=d(i+1)$ and $i$ and $i+1$ belong to the same block as an allowed move does not really present a restriction because in this case $\tau_i(\Gamma)$ would be the same as $\Gamma$, so such a move would not have any effect.}

For example, consider the $d$-decorated partition $\Gamma=\{1_{d(1)},3_{d(3)}\}\coprod\{2_{d(2)},4_{d(4)}\}$ in $\mathcal P^d_4$ with $\epsilon_{d(2)d(3)}=1$. The transposition $\tau_2$ is an allowed move and maps $\Gamma$ to $\Gamma'=\{1_{d(1)},2_{d(3)}\}\coprod\{3_{d(2)},4_{d(4)}\}$. For $\epsilon_{d(1)d(2)}=1$ the transposition $\tau_1$ is allowed and maps $\Gamma$ to $\Gamma'=\{2_{d(1)},3_{d(3)}\}\coprod\{1_{d(2)},4_{d(4)}\}$. 

Recall that a {\it{connected subset}} of $[n]$ is an interval, i.e., an arbitrary subset of consecutive integers $\{k,k+1,\dots,k+p\} \subset [n]$.

\begin{defn} A decorated partition $\Gamma  \in \mathcal P^d_n$ is called {\rm{reducible}} if there exists an allowed permutation $\sigma \in S_n$ of $\Gamma$ and an index $i\leq k$ such that $\sigma(\bar\Gamma_i)$ is a connected subset of $[n]$. We then say that $(\sigma,i)$ is a reduction and that $\Gamma$ can be reduced to the decorated partition 
$$
	\Gamma':=st\big(\sigma(\Gamma_1)\coprod\cdots \coprod\sigma(\Gamma_{i-1})\coprod \sigma(\Gamma_{i+1})\coprod\cdots\coprod \sigma(\Gamma_k)\big)
			=:(\sigma,i)(\Gamma),
$$
with $\bar\Gamma'$ being a partition of $[n - |\bar\Gamma_i|]$.
\end{defn}

\begin{defn}
A decorated partition in $\mathcal P^d_n$ is called {\rm{$\epsilon$-noncrossing}} if it can be reduced to the empty set by a sequence of reductions ${\mathbf \sigma}=(\sigma_l,i_l)\circ\dots\circ (\sigma_1,i_1).$ Such a sequence is called a {\rm{trivialization}} of $\Gamma$. The subset of $d$-decorated $\epsilon$-noncrossing partitions is denoted by ${\mathcal P}^{d,\epsilon}_{n}$.
\end{defn}

\begin{exam} {\rm{Any decorated noncrossing partition is $\epsilon$-noncrossing, $\mathcal{NC}^d_n \subset {\mathcal P}^{d,\epsilon}_{n}$. Assume that $\epsilon_{d(i)d(j)}=1$ for all $i\not=j$. Then, all $d$-decorated partitions are $\epsilon$-noncrossing. Assume that $\epsilon_{d(i)d(j)}=0$ for all pairs $(i,j)$, $i\not= j$, then a decorated partition is $\epsilon$-noncrossing if and only if it is noncrossing.}}
\end{exam}
  
\begin{defn}
Let $(A,\phi)$ be a noncommutative probability space. The $\epsilon$-cumulant $k_\epsilon$ is a map from the direct sum of tensor products $A_{d(1)}\otimes \cdots \otimes A_{d(n)}$ to $\mathbb K$, where $d$ runs over $I^{[n]}$, defined inductively by
$$
	\phi(a_1 \cdots a_n)=\sum\limits_{\Gamma \in {\mathcal P}^{d,\epsilon}_{n}}k_\epsilon^{\Gamma}({\mathbf a}),
$$
where ${\mathbf a}=(a_1,\dots,a_n)$ and ${\mathcal P}^{d,\epsilon}_{n}$ stands for the set of $d$-decorated $\epsilon$-noncrossing partitions.
\end{defn}

For example, consider ${\mathbf a}=(a_{1}, a_{2}, a_{3}, a_{4})$ in $A_{d(1)}\otimes \cdots \otimes A_{d(4)}$ with $\epsilon_{d(i)d(j)}=0$ for all pairs $(i,j)$, $i\not= j$, except $\epsilon_{d(2)d(3)}=1$. Then 
$$
	\phi(a_{1} a_{2} a_{3} a_{4})=\sum\limits_{\Gamma \in \mathcal{NC}^d_4}k_\epsilon^{\Gamma}({\mathbf a}) + k_\epsilon(a_{1} \otimes a_{3})k_\epsilon(a_{2} \otimes a_{4}).
$$

{\begin{rmk}\label{remark1}
Assume that all $a_1,\dots,a_n\in A_i$ are from the same subalgebra. Then our decoration is given by $d=(i,i,\dots,i)$ and the relevant set of $\epsilon$-noncrossing partitions ${\mathcal P}^{d,\epsilon}_n$ depends on $\epsilon$ only through our choice of $\epsilon_{ii}$. For $\epsilon_{ii}=1$ we have that ${\mathcal P}_n^{d,\epsilon}$ are just all partitions and the defining equation for $\epsilon$-cumulants reduces to the classical moment-cumulant relation; hence in this case the $\epsilon$-cumulants are the same as classical cumulants. For $\epsilon_{ii}=0$ we have that ${\mathcal P}^{d,\epsilon}_n$ are exactly the noncrossing partitions and the defining equation for $\epsilon$-cumulants reduces to the free moment-cumulant relation; hence in this case the $\epsilon$-cumulants are the same as free cumulants. For arguments from different subalgebras the situation gets of course more complicated; the $\epsilon$-cumulants will then depend also on the non-diagonal values of the $\epsilon_{ij}$ and are in general neither classical nor free cumulants.
\end{rmk}}


\section{The lattice of $\epsilon$-noncrossing partitions}
\label{sect:lattice}

Let us show first that $\epsilon$-noncrossing partitions are stable by taking subpartitions. The property may be clear intuitively, we simply indicate its proof.

\begin{lem}
Let $\Gamma$ be a $d$-decorated $\epsilon$-noncrossing partition and $i\leq n$. Consider the induced decorated partition of $[n]-\{i\}$ and write $\Gamma'$ for its standardization, with decoration $d'$ induced by the standardization of the restriction of $d$ to $[n]-\{i\}$. Then, $\Gamma'$ is a $d'$-decorated $\epsilon$-noncrossing partition.
\end{lem}

\begin{proof}
Indeed, let $\tau_j$ be an allowed move for the decorated partition $\Gamma$. It maps the pair $(\Gamma,\{i\})$ to $(\tau_j(\Gamma),\{i\})$ if $j<i-1$ or $j>i$; to $(\tau_j(\Gamma),\{i-1\})$ if $j=i-1$; and to $(\tau_j(\Gamma),\{i+1\})$ if $j=i$.
It induces on $\Gamma'$:
\begin{itemize}
	\item the allowed move $\tau_j$ if $j<i-1$, 

	\item the identity map if $j=i-1$ or $i$, 

	\item the allowed move $\tau_{j-1}$ if $j>i$. 
\end{itemize}
By this process, an allowed permutation (respectively~reduction, respectively~trivialization) of $\Gamma$ induces an allowed permutation (respectively~reduction, respectively~trivialization) of $\Gamma'$ (the permutation maps $\Gamma'=st(\Gamma_{|[n]-\{i\}})$ to $st(\sigma(\Gamma)_{|[n]-\{\sigma(i)\}})$). The lemma follows.
\end{proof}

\begin{cor}
Using the same notation as in the foregoing lemma, for an arbitrary subset $X \subset [n]$, the standardization $\Gamma_X$ of the restriction of $\Gamma$ to $X$ is an $\epsilon$-noncrossing partition.
\end{cor}

For later use, if $\beta$ is an allowed permutation of $\Gamma$, then we write $\beta_X$ for the induced allowed permutation of $\Gamma_X$. Analogously, for a reduction $(\sigma, i)$ (respectively trivialization $\mathbf \sigma$) of $\Gamma$, the induced reduction (trivialization) is denoted $(\sigma_X,i)$ (respectively ${\mathbf \sigma}_X$). Notice that the induced reduction can be trivial, i.e., simplify to the identity map, if $\Gamma_i \subset [n]-X$.

Let us consider now a decoration $d$ and the set $\mathcal P^d_n$  of $d$-decorated partitions of $[n]$. The refinement of partitions defines a preorder. It is actually a lattice for which the meet $\Gamma\wedge\Gamma'$ is obtained by taking the various intersections of the blocks of the two partitions $\Gamma$ and $\Gamma'$. 

The subset ${\mathcal P}^{d,\epsilon}_{n}$ of $d$-decorated $\epsilon$-noncrossing partitions inherits from $\mathcal P_n^d$ a poset structure. It has a maximal element (the trivial $d$-decorated  partition $\hat{1}_n=[n]$) and a minimal $d$-decorated  element (the full partition, $\hat{0}_n=\{1\} \coprod \cdots \coprod \{n\}$). 

\begin{prop}
The meet in $\mathcal P^d_n$ is also the meet in $\mathcal P^{d,\epsilon}_n$. Moreover, $\mathcal P^{d,\epsilon}_n$ is a lattice.
\end{prop}

For notational simplicity, we will omit taking standardizations in the proof, when considering subpartitions.

\begin{proof}
Let $\Gamma = \Gamma_1\coprod \cdots \coprod \Gamma_k$ and $\Gamma' = \Gamma'_1 \coprod \cdots \coprod \Gamma'_l$ be two $\epsilon$-noncrossing partitions -- with respect to the same decoration $d$ -- and $(\sigma,i)$ be a reduction of $\Gamma$, and $\mathbf \beta$ be a trivialization of $\Gamma'$. The meet of $\Gamma$ and $\Gamma'$ in $\mathcal P^{d}_n$ is $\Gamma \wedge \Gamma' = (\Gamma_1 \cap \Gamma'_1) \coprod (\Gamma_1 \cap \Gamma'_2) \coprod \cdots \coprod (\Gamma_k \cap \Gamma'_l)$.

The allowed permutation $\sigma$ of $\Gamma$ is also an allowed permutation of $\Gamma \wedge \Gamma'$ (indeed, an allowed move $\tau_i$ for $\Gamma$ is such that $\epsilon_{d(i)d(i+1)}=1$ and either $d(i)\not= d(i+1)$ or $d(i)=d(i+1)$ and $i,i+1$ belong to two different blocks of $\Gamma$. These properties are inherited by $\Gamma\wedge \Gamma'$). The permutation $\sigma$ maps $\Gamma \wedge \Gamma'$ to the decorated partition with underlying partition
$$
	\overline{\sigma(\Gamma\wedge\Gamma')}=\overline{\sigma (\Gamma_1\cap\Gamma'_1)}\coprod
	\cdots\coprod \overline{\sigma (\Gamma_k\cap\Gamma'_l)},
$$
where $\overline{\sigma(\Gamma_i)}=\overline{\sigma(\Gamma_i\cap\Gamma'_1)}\coprod\cdots\coprod \overline{\sigma(\Gamma_i\cap\Gamma'_l)}$. Since $(\sigma, i)$ is a reduction for $\Gamma$, $\overline{\sigma(\Gamma_i)}$ is an interval, i.e., a connected subset of $[n]$. We finally get that ${\mathbf\beta}_{\Gamma_i}\circ{\sigma^{-1}}_{\sigma(\Gamma_i)}$ is a trivialization of the partition $\sigma(\Gamma_i\cap\Gamma'_1)\coprod \cdots \coprod \sigma(\Gamma_i\cap\Gamma'_l)$ of $\sigma(\Gamma_i)$. 

Up to an allowed permutation, we have shown that there exists a sequence of reductions from $\Gamma\wedge\Gamma'$ to $st(\Gamma_{|[n]-\Gamma_i})\wedge st(\Gamma'_{|[n]-\Gamma_i})$. The first part of the Proposition follows by induction on $[n]$. 

By standard results in lattice theory, a poset in which every pair of elements has a meet, and which has a maximal element is a lattice \cite[Prop. 3.3.1]{Stanley}. The statement in the proposition follows.
\end{proof}


\section{Values of $\epsilon$-cumulants.}
\label{sect:values}

Recall that, by definition of $\epsilon$-cumulants, and due to our conventions for units, we have $k_\epsilon (1_j)=k_\epsilon (1_A)=\phi(1_A)=1$.

\begin{lem}
Consider ${\mathbf a}=(a_1,\dots,a_n)$, $a_j\in A_{i_j}$. Then, for $n>1$, $k_\epsilon (a_1,\dots,a_n)=0$ if one of the $a_j$ is $1_{i_j}$.
\end{lem}

\begin{proof}
The proof runs by induction on $n$. For $n=2$, we have $k_\epsilon (a_1, 1_{i_2})+k_\epsilon (a_1)=k_\epsilon (a_1, 1_{i_2})+k_\epsilon (a_1)k_\epsilon (1_{i_2})=\phi(a_11_{i_2})=\phi(a_1)=k_\epsilon (a_1)$, and therefore $k_\epsilon (a_1,1_{i_2})=0$. Similarly, $k_\epsilon (1_{i_1},a_2)=0$.

In general, assuming that the property holds up to order $n$, we have, for an arbitrary $p\in I$, $\phi(a_1\cdots a_i1_pa_{i+1}\cdots a_n)=\phi(a_1\cdots a_n)$. The term on the righthand side can be expanded as a sum of cumulants $k_\epsilon^{\Gamma}({\mathbf a})$, where $\Gamma$ runs over $\epsilon$-noncrossing partitions associated to the decoration $(d(1), \dots,d(n))=(i_1,\dots,i_n)$.  The term on the lefthand side can be expanded as a sum of cumulants $k_\epsilon^{\Gamma}({\mathbf a})$, where $\Gamma$ runs over $\epsilon$-noncrossing partitions associated to the decoration $(d(1),\dots,d(i),p,d(i+1), \dots,d(n))$. These partitions split into three groups:
\begin{itemize}
	\item The maximal partition $\hat1_{n+1}=[n+1]$, corresponding to the cumulant $k_\epsilon (a_1\otimes \cdots \otimes a_i  \otimes 1_p  \otimes a_{i+1} \otimes \cdots  \otimes a_n)=k_\epsilon (a_1, \ldots,1_p,\ldots  a_n)$.
	
	\item Partitions $\Gamma\not= \hat1_{n+1}$ where $i+1$ (the index of $1_p$ in the sequence $(a_1,\dots , a_i,1_p,a_{i+1},\dots , a_n)$) is contained in a block with at least two elements. By the induction hypothesis, the corresponding cumulants $k_\epsilon^{\Gamma}(a_1,\dots , a_i,1_p,a_{i+1},\dots , a_n)$ are equal to zero.

	\item Partitions $\Gamma$ where $\{i+1\}$ is a singleton block. Those partitons are in bijection with the $\epsilon$-noncrossing partitions associated to the decoration $(d(1), \ldots,d(n))$, and we have that $k_\epsilon^\Gamma(a_1,\ldots , a_i, 1_p , a_{i+1}, \ldots , a_n)=k_\epsilon^{\Gamma_{[n]-\{i+1\}}}(a_1,\dots , a_n)$.
\end{itemize}
The statement of the lemma follows.
\end{proof}

Consider now the following problem: let ${\mathbf a}=(a_1,\dots ,a_n)\in A_{i_1}\times\dots \times A_{i_n}$  with associated decoration map $\tilde d(j)=i_j$. Choose a sequence of integers $1\leq p_1< \dots < p_{m-1}< n$ such that $i_1=\cdots =i_{p_1},\ldots , i_{p_{m-1}+1}=\cdots =i_n$. Define now ${\mathbf b}=(b_1,\ldots ,b_m)$ by $b_j:=a_{i_{p_{j-1}+1}}\cdots a_{i_{p_j}}$, so that $b_j\in A_{i_{p_j}}$ (with $p_0:=0$, $p_m:=n$). We want to express the cumulant $k_\epsilon (b_1,\dots, b_m)$ in terms of cumulants $k_\epsilon^\Gamma ({\mathbf a}).$

Let $s$ be the surjection from $[n]$ to $[m]$ defined by: $s(1)=\dots =s(p_1):=1,\dots ,s(p_{m-1}+1)=\dots =s(n):=m$.

\begin{lem}\label{lem5}
Let $\Gamma=\Gamma_1\coprod \cdots \coprod\Gamma_k$ be a $d$-decorated $\epsilon$-noncrossing partition of $[m]$. Define $\tilde\Gamma$ as the decorated partition of $[n]$ obtained as the inverse image of $\Gamma$ under the surjection $s$, with decoration $\tilde d$, where $\tilde d(i):=d\circ s(i)$. Then, $\tilde\Gamma$ is an $\epsilon$-noncrossing partition.
\end{lem}

\begin{proof}
Indeed, let $\tau_i$ be an allowed move for $\Gamma$. Then, since $\epsilon_{d(i)d(i+1)}=1$, we have $\epsilon_{\tilde d(a)\tilde d(b)} = 1$ for any $a$ in $s^{-1}(\{i\})$, $b$ in $s^{-1}(\{i+1\})$. When furthermore $d(i)=d(i+1)$, $i$ and $i+1$ belong to different blocks of $\Gamma$ and the same property holds for $a$, $b$ and $\tilde\Gamma$. It follows that the permutation $\tilde\tau_i$ in $S_n$ switching $\{p_{i-1}+1,\dots ,p_{i}\}$ and $\{p_{i}+1,\dots ,p_{{i+1}}\}$ is an allowed permutation of $\tilde\Gamma$. Therefore, an allowed permutation, respectively a trivialization of $\Gamma$ induces an allowed permutation, respectively a trivialization of $\tilde\Gamma$. The Lemma follows.
\end{proof}

\begin{lem}
The map $s^{-1}:\Gamma\mapsto \tilde\Gamma$ defines a bijection between the lattice ${\mathcal P}^{d,\epsilon}_m$ of $d$-decorated $\epsilon$-noncrossing partitions of $[m]$ and the interval sublattice $[\tilde {0}_m,\tilde { 1}_m]$ of the lattice ${\mathcal P}^{\tilde{d},\epsilon}_{n}$ of $\tilde d$-decorated $\epsilon$-noncrossing partitions of $[n]$, where we set $\tilde{0}_m:=\tilde{\hat 0}_m$ and $\tilde{1}_m:=\tilde{\hat 1}_m = \hat{1}_n$.
\end{lem}

\begin{proof}In view of Lemma \ref{lem5},
the Lemma follows if we prove that all $\tilde d$-decorated $\epsilon$-noncrossing partitions $\Phi$ of $[n]$ with $\Phi\geq \tilde 0_m$ are of the form $\tilde \Gamma$. However, since $\Phi \geq \tilde 0_m$, $\Phi=s^{-1}(\Gamma)$ for a certain $d$-decorated partition $\Gamma$ of $[m]$. Let us show that $\Gamma$ is $\epsilon$-noncrossing.

Let, for that purpose, $\mathbf \sigma$ be a trivialization of $\Phi$. Set $X:=\{1,p_1+1,\dots ,p_{m-1}+1\}$. Then $X$ is in bijection with $[m]$ and $\Phi_X=\Gamma$ as a $d$-decorated partition. The Lemma follows since $\mathbf \sigma_X$ is a trivialization of $\Phi_X$.
\end{proof}

\begin{cor}
Let $f_m$ and $g_m$ be functions on ${\mathcal P}^{d,\epsilon}_m$ related by 
$$
	f_m(\Gamma)=\sum\limits_{\Phi \in {\mathcal P}^{d,\epsilon}_m \atop \Phi \leq \Gamma} g_m(\Phi)
$$
and let $f_n,g_n$ be two functions on ${\mathcal P}^{\tilde{d},\epsilon}_{n}$ related by
$$
	f_n(\Gamma)=\sum\limits_{\Phi\in{\mathcal P}^{\tilde{d},\epsilon}_{n} \atop \Phi \leq \Gamma} g_n(\Phi).
$$
Then, the following statements are equivalent:
\begin{itemize}

\item We have, for all $\Gamma\in{\mathcal P}^{d,\epsilon}_m$ that $f_m(\Gamma)=f_n(\tilde\Gamma),$

\item We have, for all $\Gamma\in{\mathcal P}^{d,\epsilon}_m$ that $g_m(\Gamma)=\sum\limits_{\Phi\in{\mathcal P}^{\tilde{d},\epsilon}_n \atop \Phi\vee\tilde{ 0}_m = \tilde\Gamma} g_n(\Phi).$

\end{itemize}
\end{cor}

The Corollary follows from \cite[Thm.~3.4]{speicher_2000}.

Applying these results to the functions $f_m(\Gamma):=\phi^\Gamma(\mathbf b)$ on ${\mathcal P}^{d,\epsilon}_m$, respectively $f_n(\Gamma):=\phi^\Gamma(\mathbf a)$ on ${\mathcal P}^{\tilde{d},\epsilon}_{n} $ (hence the $g$'s are the $\epsilon$-cumulants) and noticing that $f_m(\Gamma)=\phi^\Gamma({\mathbf b})=\phi^{\tilde\Gamma}(\mathbf a)=f_n(\tilde\Gamma)$, we get the next result. 

\begin{thm}\label{T8}
With the above notations, let $\Gamma$ be a $d$-decorated partition, then the following equation holds:
$$
	k_\epsilon^\Gamma(b_1,\dots,b_m)=\sum\limits_{\Phi\in{\mathcal P}^{\tilde{d},\epsilon}_{n}  \atop  
	\Phi\vee \tilde{ 0}_m=\tilde \Gamma}k_\epsilon^\Phi(a_1,\dots ,a_n).
$$
\end{thm}

For further use, we will say that $\Phi$ with $\Phi\vee \tilde{0}_m=\tilde \Gamma$ couples two blocks of $ \tilde{ 0}_m$ if and only if they are subsets of the same block of $\tilde\Gamma$.


\section{Vanishing of cumulants and independence}
\label{sect:vanishing}

This section is dedicated to the proof of the main theorem, which states that $\epsilon$-independence is equivalent to the vanishing of mixed cumulants.

\begin{thm}\label{mainT}
Let $(A,\phi)$ be a noncommutative probability space, $I$ a set of indices, $\epsilon_{ij}=\epsilon_{ji}\in\{0,1\}$, and $(A_i)_{i\in I}$ unital subalgebras with the same unit as $A$; 
{and such that algebras $A_i$ and $A_j$, for $i\not= j$, commute whenever $\epsilon_{ij}=1$.} 
Then, the following statements are equivalent:
\begin{itemize}
	\item {\rm{(i)}} The $A_i$ are $\epsilon$-independent.

	\item {\rm{(ii)}} For all $n\geq 2$ and $a_j\in A_{i_j}$ we have $k_\epsilon (a_1,\dots,a_n)=0$, 
	whenever there exist $1\leq l<k\leq n$ with $i_l\not= i_k$.
\end{itemize}
\end{thm}

{Let us point out again that statement ${\rm{(i)}}$ does not depend on the choice of
the $\epsilon_{ii}$, whereas statement ${\rm{(ii)}}$ uses this information, as for different choices of the diagonal values we have different versions of the $k_\epsilon$.}

We first show two technical lemmas.

\begin{lem}\label{L1}
Let $\Gamma$ be a decorated partition and $(a_1,\dots ,a_n)\in A_{d(1)}\times \cdots \times A_{d(n)}$. Then, if $\tau_i$ is an allowed move for $\Gamma$ and $\Gamma'=\tau_i(\Gamma )$,
$$
	\phi^\Gamma (a_1,\ldots ,a_n)=\phi^{\Gamma'}(a_1,\ldots, a_{i+1},a_i,\ldots,a_n)
$$
and
$$
	k_\epsilon^\Gamma (a_1,\ldots ,a_n)=k_\epsilon^{\Gamma'}(a_1,\ldots, a_{i+1},a_i,\ldots,a_n).
$$
\end{lem}

\begin{proof}
When $i$ and $i+1$ belong to different blocks, the  identities follow at once from the definitions. When this is not the case, $d(i)\not= d(i+1)$ and we have the commutativity relation $a_ia_{i+1}=a_{i+1}a_i$. Under this assumption, the first identity follows then from the properties of $\phi$. The second follows by induction from the identity $\phi(a_1\cdots a_n)=\phi(a_1\cdots a_{i+1}a_i\cdots a_n)$ together with the identity relating $\phi$ and $k_\epsilon$ (using the induction hypothesis stating that  $k_\epsilon (a_1,\ldots ,a_n)$ equals $k_\epsilon (a_1,\dots, a_{i+1},a_i,\dots,a_n)$, and noticing that it implies that $k_\epsilon^\Gamma (a_1, \ldots , a_m) = k_\epsilon^{\Gamma'}(a_1, \ldots, a_{i+1},a_i,\ldots, a_m)$ for all decorated partitions $\Gamma$ whose blocks contain at most $n$ elements).
\end{proof}

The following lemma follows from the definitions.

\begin{lem}\label{L2}
Assume that $d=(i_1,\ldots,i_n) \in I_n^\epsilon$, then for any $d$-decorated partition $\Gamma$ and allowed move $\tau_k$ for $\Gamma$ we also have that $(i_1,\ldots,i_{k+1}, i_k, \ldots, i_n)$ is in $I_n^\epsilon$. As a consequence, $I_n^\epsilon$ is stable under the action of allowed permutations.
\end{lem}

\begin{rmk} {\rm{ Note that allowed permutations do not form a group due to the fact that whether a permutation is  allowed or not depends on the decorated partition under consideration. Decorations in $I_n^\epsilon$ and allowed permutations do form a {\it{groupo\"\i d}}, and ``stability'' has to be understood in that sense, that is, an allowed permutation with respect to a given $d$-decorated partition with $d\in I_n^\epsilon$ maps this decoration to another element of $I_n^\epsilon$.}}
\end{rmk}

\begin{proof}[Proof of Thm.~\ref{mainT}] Let us prove Theorem \ref{mainT} by showing first that (i) implies (ii) using induction. By Lemma 4, we can always assume that the $a_i$ are centered ($\phi(a_i)=0$). For $n=2$ and $i_1\not= i_2$, we have, due to $\epsilon$-independence, 
$$
	0=\phi(a_1a_2)=k_\epsilon (a_1,a_2)+k_\epsilon (a_1)k_\epsilon (a_2)=k_\epsilon (a_1,a_2).
$$
Let us assume that the property holds till rank $n-1\geq 2$.

Step 1. Let us assume first that {$d=(i_1,\dots,i_n) \in I_n^\epsilon$}. From $\epsilon$-independence we have
$$
	0={ \phi(a_1\cdots a_n)}=k_\epsilon (a_1,\dots,a_n)+\sum\limits_{\Gamma \in {\bar{\mathcal P}}^{{d},\epsilon}_{n}}k_\epsilon^\Gamma({\mathbf a}),
$$
where ${\bar{\mathcal P}}^{{d},\epsilon}_{n}$ stands for the set of nontrivial $d$-decorated $\epsilon$-noncrossing partitions, i.e., ${{\mathcal P}}^{{d},\epsilon}_{n} -\{\hat1_n\}$. Let us show that $k_\epsilon^\Gamma({\mathbf a})=0$ for all $\Gamma\in{\bar{\mathcal P}}^{{d},\epsilon}_{n}$, which will imply the property. Since an $\epsilon$-noncrossing partition $\Gamma$ can be transformed into an $\epsilon$-noncrossing partition $\Gamma'$ containing a block $\Gamma'_i$, which is an interval, by an allowed permutation, Lemma \ref{L1} and Lemma \ref{L2} above imply that we can always assume, up to an allowed permutation of the $a_i$, that $\Gamma$ has this property. This interval is therefore either a singleton or, by the assumption {$d=(i_1,\dots,i_n) \in I_n^\epsilon$}, it contains at least two elements with different decorations. The vanishing of $k_\epsilon^\Gamma(\mathbf a)$ follows then from the induction hypothesis.

Step 2. Assume that there exist distinct $k, l$ with $i_k \not= i_l$. We have to show that $k_\epsilon (a_1 , \dots , a_n ) = 0$. Lemmas \ref{L1} and \ref{L2} imply that we do not change the problem if we let an allowed permutation reorder the $a_i$. In particular, we can permute the $a_i$ in such a way that $a_1,\dots,a_{p_1}$ belong to the same algebra $A_{i_{p_1}}$, $a_{p_1+1},\dots,a_{p_2}$ belong to the same algebra $A_{i_{p_2}}$, and so on, until $a_{p_{m-1}+1},\dots,a_{n}$ belong to the same algebra $A_{i_{n}}$, so that we also have $(i_{p_1},\dots, i_n)\in I_m^\epsilon$. Note that $m \geq 2$
because of our assumption $i_k \not= i_l$. We set $b_1:=a_1\cdots a_{p_1}$, $b_2:=a_{p_1+1}\cdots a_{p_2}$, and so on, until $b_m:=a_{p_{m-1}+1}\cdots a_{n}.$ Then, by Theorem \ref{T8}, we have
 $$
 	k_\epsilon (b_1,\dots,b_m)=\sum\limits_{\Phi \in \mathcal{P}^{\tilde{d},\epsilon}_n \atop \Phi\vee \tilde 0_m=[n]}k_\epsilon^\Phi(a_1,\dots ,a_n).
$$
Since $(i_{p_1},\dots, i_n)\in I_m^\epsilon$ we have $k_\epsilon (b_1,\dots,b_m)=0$ according to Step 1. Furthermore, for $\Phi\not= [n]$, the term $k_\epsilon^\Phi(a_1,\ldots ,a_n)$ is a product of cumulants of lengths strictly smaller than $n$. Thus our induction hypothesis applies to them and we see that $k_\epsilon^\Phi(a_1,\ldots ,a_n)$ can be different from zero only if each block of $\Phi$ contains only elements from the same subalgebra. However, we are only looking at $\Phi$ with the additional property that $\Phi\vee \tilde 0_m=[n]$, which means that $\Phi$ has to couple all blocks of $\tilde 0_m$. This is possible only if a block of $\Phi$ contains two elements in different algebras, hence a contradiction if we had $k_\epsilon^\Phi(a_1,\ldots ,a_n)\not= 0$. Thus there is no non-vanishing contribution in the above sum from the non trivial partitions in the right hand side and we get $k_\epsilon (a_1 , \ldots , a_n ) = 0$.
 
The direction (ii) implies (i) is simpler: Consider $a_j \in A_{i_j}, j = 1, \dots, n$ with $(i_1 ,\ldots,i_n)\in I_n^\epsilon$ and $\phi(a_j) = 0$,  for $j = 1, \ldots, n $. Then we have to show that $\phi (a_1 \cdots a_n ) = 0$. However, we have $\phi (a_1 \cdots a_n )=\sum_{\Gamma}k_\epsilon^\Gamma (a_1,\ldots,a_n)$ where $\Gamma$ runs over all $\epsilon$-noncrossing partitions associated to the decoration $(i_1 ,\dots,i_n)$. By Lemmas \ref{L1} and \ref{L2}, to show that $k_\epsilon^\Gamma (a_1,\ldots,a_n)=0$ we can always assume that $\Gamma$ contains a block which is an interval. Making this assumption, each product $k_\epsilon^\Gamma (a_1,\ldots,a_n)=\prod_{i=1}^k k_\epsilon (a_{\Gamma_i})$ contains at least one factor of the form $k_\epsilon (a_l , a_{l+1} , \ldots , a_{l+p} )$ which vanishes, {for $p>0$, because of our assumption on the vanishing of mixed cumulants and, for $p=0$, by the assumption that $\phi(a_l)=0$.}
\end{proof}

{By taking also into account our Remark \ref{remark1}, the vanishing of mixed cumulants allows us to recover directly the
main result from \cite{speichWyso_2016}. Note that we have here actually 
the more general version of this as alluded to in Remark 5.4 of  \cite{speichWyso_2016}. 

For $d\in I^{[n]}$ we write $\ker d$ for the partition $\coprod\limits_{i\in I}d^{-1}(\{i\})$ (i.e. the one whose blocks are the subsets of $[n]$ of elements with the same decoration). Notice that for a $d$-decorated partition $\Gamma$, $\Gamma\leq \ker d$ means precisely that elements in an arbitrary block of $\Gamma$ belong to the same subalgebra.

\begin{cor}
Let $(A,\phi)$ be a noncommutative probability space and assume that subalgebras $A_i \subset A$, $i \in I$, are {\rm{$\epsilon$-independent}}. Consider $d\in I^{[n]}$ and $a_j\in A_{d(j)}$. Then we have for the mixed moment
$$
	\phi(a_1\cdots a_n) = \sum\limits_{\Gamma \in {\mathcal P}^{d,\epsilon}_{n} \atop \Gamma\leq \ker d} 
	k_\epsilon^{\Gamma}({\mathbf a}),
$$
where
$k_\epsilon^{\Gamma}({\mathbf a})$ is a product of classical and free cumulants, according to the blocks of $\Gamma$ and our choice of diagonal elements $\epsilon_{ii}$ for the $i$-value of the block; a block with $\epsilon_{ii}=1$ contributes a classical cumulant as factor, a block with $\epsilon_{ii}=0$ contributes a free cumulant as factor. 
\end{cor}
}


\end{document}